\documentclass[a4paper]{amsart}
\usepackage[utf8]{inputenc}
\usepackage{amssymb, mathtools, hyperref, cases, enumerate, xcolor, cancel}
\usepackage[shortlabels]{enumitem}

\usepackage[
	backend=biber,
	style=alphabetic,
]{biblatex}
\AtEveryBibitem{\clearlist{language}}
\numberwithin{equation}{section}
\newtheorem{thm}{Theorem}[section]
\newtheorem{prop}[thm]{Proposition}
\newtheorem{lem}[thm]{Lemma}
\newtheorem{cor}[thm]{Corollary}
\newtheorem{rmk}[thm]{Remark}
\theoremstyle{definition}

\newcommand{\nat}{\mathbb{N}}
\newcommand{\real}{\mathbb{R}}
\newcommand{\complex}{\mathbb{C}}

\newcommand{\eps}{\varepsilon}

\newcommand{\E}{\mathcal{E}}
\newcommand{\G}{\mathcal{G}}
\newcommand{\op}{\mathcal{H}}

\newcommand{\dif}{\mathrm{d}}

\DeclareMathOperator{\Dom}{Dom}
\DeclareMathOperator{\divergence}{div}

\newcommand{\HartreeDelta}{
i \partial_t \psi
=
- \Delta_x \psi
+
\alpha \delta_0 \psi
-
\parens{I_\beta \ast \abs{\psi}^p} \abs{\psi}^{p - 2} \psi
}

\DeclarePairedDelimiter{\abs}{\lvert}{\rvert}
\DeclarePairedDelimiter{\norm}{\lVert}{\rVert}
\DeclarePairedDelimiter{\parens}{(}{)}
\DeclarePairedDelimiter{\set}{\{}{\}}
\DeclarePairedDelimiter{\brackets}{\lbrack}{\rbrack}

\DeclarePairedDelimiter{\angles}{\langle}{\rangle}
\DeclarePairedDelimiter{\cci}{\lbrack}{\rbrack}
\DeclarePairedDelimiter{\coi}{\lbrack}{\lbrack}

\DeclarePairedDelimiter{\ooi}{\rbrack}{\lbrack}

\title[Hartree-type equations in $\real^3$ with a delta potential]{On the existence of ground states to Hartree-type equations in $\real^3$ with a delta potential}
\begin{document}

\author{Gustavo de Paula Ramos}
\address{Instituto de Matemática e Estatística, Universidade de São Paulo, Rua do Matão, 1010, 05508-090 São Paulo SP, Brazil}
\email{gpramos@ime.usp.br}
\urladdr{http://gpramos.com}

\begin{abstract}
Consider the Hartree-type equation in $\real^3$ with a delta potential formally described by
\[\HartreeDelta\]
where $\alpha \in \real$; $0 < \beta < 3$ and we want to solve for
$\psi \colon \real^3 \times \real \to \complex$. By means of a Pohožaev identity, we show that if $p = \parens{3 + \beta} / 3$ and $\alpha \geq 0$, then the problem has no ground state at any mass $\mu > 0$. We also prove that if
\[
\frac{3 + \beta}{3}
<
p
<
\min \parens*{
	\frac{5 + \beta}{3}, \frac{5 + 2 \beta}{4}
},
\]
which includes the physically-relevant case $p = \beta = 2$, then the problem admits a ground state at any mass $\mu > 0$.

\end{abstract}

\date{\today}
\maketitle

\section{Introduction}

\subsection{Context and motivation}

Consider the Hartree-type equation in $\real^3$ with a delta potential, a.k.a. \emph{delta interaction}, \emph{point interaction} or \emph{point defect}, formally described by
\begin{equation}
\label{intro:eqn:HartreeDelta}
\HartreeDelta,
\end{equation}
where $p > 1$; $\alpha \in \real$; $0 < \beta < 3$;
$I_\beta \colon \real^3 \setminus \set{0} \to \real$ denotes the \emph{Riesz potential} given by
\[
I_\beta \parens{x}
:=
\frac{1}{\Gamma \parens{\beta / 2} \pi^{3 / 2} 2^\beta}
\Gamma \parens*{\frac{3 - \beta}{2}}
\frac{1}{\abs{x}^{3 - \beta}}
\]
and we want to solve for
$\psi \colon \real^3 \times \real \to \complex$.

The \emph{Hartree-type equation}
\begin{equation}
\label{intro:eqn:Hartree}
i \partial_t \psi
=
-
\Delta_x \psi
-
\parens*{I_\beta \ast \abs{\psi}^p} \abs{\psi}^{p - 2} \psi
\end{equation}
appears in physics as a model for self-interaction in a number of different contexts. The case $p = \beta = 2$ is specially important, being used, for instance, to model self-gravitating boson stars, in which context it is usually called the \emph{Hartree equation} or \emph{Schrödinger--Newton equation}. For more details about related physical models, we refer the reader to \cite[Section 2]{morozGuideChoquardEquation2017}.

Over the last years, there has been an increasing interest for semilinear PDEs of second order with a delta potential in the whole space $\real^3$. Let us begin by describing recent developments for these equations in presence of local nonlinearities. The evolution problem for the NLSE formally described by
\[
\begin{cases}
i \partial_t \psi
=
-\Delta_x \psi + \alpha \delta_0 \psi \pm \abs{\psi}^{r - 1} \psi;
\\
\psi \parens{\cdot, 0} = \psi_0
\end{cases}
\]
with $2 < r < 3$ was considered in \cite{cacciapuotiWellPosednessNonlinear2021, cacciapuotiFailureScatteringNLSE2023}, while the existence and radiality of ground states in the focusing case
\[
i \partial_t \psi
=
-\Delta_x \psi + \alpha \delta_0 \psi - \abs{\psi}^{r - 1} \psi
\]
was established in \cite{adamiExistenceStructureRobustness2022}.

As for nonlocal nonlinearities, the Cauchy problem in $\real^3$ formally described by
\[
\begin{cases}
i \partial_t \psi
=
-
\Delta_x \psi
+
\alpha \delta_0 \psi
+
\parens*{w \ast \abs{\psi}^2} \psi;
\\
\psi \parens{0, \cdot} = \psi_0
\end{cases}
\]
was recently addressed in \cite{michelangeliSingularHartreeEquation2021, dugandzijaSingularSolutionHartree2023}, both considering the case $\alpha \geq 0$. In \cite{michelangeliSingularHartreeEquation2021}, Michelangeli, Olgiati \& Scandone gave sufficient conditions for local/global well-posedness in different functional spaces according to the regularity of $w$, including $L^2$ or $H^s$ for
$0 < s < 1 / 2$, (see \cite[Theorems 1.1--1.6]{michelangeliSingularHartreeEquation2021}). In \cite{dugandzijaSingularSolutionHartree2023}, Dugandžija and Vojnović used the framework of Colombeau algebras to obtain a similar result when $a :=  \psi_0 \in H^2 \cap H^2_\alpha$ is odd and $w \in L^\infty$ is even, where $H^2_\alpha$ denotes a fractional perturbed Sobolev space.

To the best of our knowledge, there are no studies establishing sufficient conditions for the existence of ground states of \eqref{intro:eqn:HartreeDelta} in $\real^3$, hence the motivation for this note. We remark that the situation of stationary states of \eqref{intro:eqn:HartreeDelta} in $\real^2$ under the hypothesis
$p = \beta = 2$ and a previously known frequency $\omega$ was studied by Georgiev, Michelangeli \& Scandone in \cite{georgievStandingWavesGlobal2022}, which is different from the problem considered in this note.

\subsection{Technical framework}

Let us recall the setting considered in \cite{adamiExistenceStructureRobustness2022} under which we can define the rigorous version of \eqref{intro:eqn:HartreeDelta}.

Given
$\lambda > 0$, we let $\G_\lambda$ denote the fundamental solution of $-\Delta + \lambda$, which is given by
\[
\G_\lambda \parens{x}
=
\frac{1}{4 \pi \abs{x}} e^{- \sqrt{\lambda} \abs{x}}
\]
for every $x \in \real^3 \setminus \set{0}$. Notice that
$\G_\lambda \not \in H^1$ and
\begin{equation}
\label{intro:eqn:integrabilityOfGlambda}
\G_\lambda \in L^r
\quad \quad \text{if, and only if,} \quad \quad
1 \leq r < 3.
\end{equation}
Suppose that $\lambda \neq \mu$ in $\ooi{0, \infty}$. An explicit computation shows that $\G_\lambda - \G_\mu \in H^1$ (for the explicit expression of $\norm{\G_\lambda - \G_\mu}_{H^1}$, see Lemma \ref{prelim:lem:Glambda}) and
\[
\Delta \parens{\G_\lambda - \G_\mu}
=
\lambda \parens{\G_\lambda - \G_\mu}
\quad \text{in} \quad
L^2,
\]
so $H^2$-elliptic regularity implies
\begin{equation}
\label{intro:eqn:GlambdaMinusGnuIsInH2}
\G_\lambda - \G_\mu \in H^2
\quad \quad \text{for every} \quad \quad
\lambda \neq \mu
\quad \quad \text{in} \quad \quad
\ooi{0, \infty}.
\end{equation}

The operator $\op \colon \Dom \parens{\op} \to L^2$ will denote the self-adjoint realization of $- \Delta + \alpha \delta_0$, which acts as
\[
\op \brackets{\phi + q \G_\lambda}
=
- \Delta \phi - q \lambda \G_\lambda
\]
for every $\parens{\phi, q, \lambda} \in T$, where
\[
T := \set*{
	\parens{\phi, q, \lambda}
	\in
	H^2 \times \complex \times \ooi{0, \infty}:
	\phi \parens{0}
	=
	\parens*{
		\alpha + \frac{\sqrt{\lambda}}{4 \pi}
	}
	q
}
\]
and
\[
\Dom \parens{\op}
:=
\set*{
	\phi + q \G_\lambda : \parens{\phi, q, \lambda} \in T
}
\subset
L^2.
\]
In contrast with the embeddings $H^1 \hookrightarrow L^r$ for $2 \leq r \leq 6$, it follows from \eqref{intro:eqn:integrabilityOfGlambda} that $\Dom \parens{\op} \subset L^r$ if, and only if, $2 \leq r < 3$.

Notice that the mapping
\[
T \ni \parens{\phi, q, \lambda}
\mapsto
\phi + q \G_\lambda \in \Dom \parens{\op}
\]
is not injective. More precisely, it follows from \eqref{intro:eqn:GlambdaMinusGnuIsInH2} that given
$u \in \Dom \parens{\op}$, there exists a unique $q \in \complex$, called the \emph{charge} of $u$, for which we can associate each
$\lambda > 0$ to a $\phi_\lambda \in H^2$ such that
$u = \phi_\lambda + q \G_\lambda$. This nonuniqueness of representation of functions $u \in \Dom \parens{\op}$ in terms of triples $\parens{\phi, q, \lambda} \in T$ will play a major role in the arguments.

If
\begin{equation}
\label{intro:eqn:conditionForChoquardL2}
\frac{3 + \beta}{3}
\leq
p
<
\frac{5 + 2 \beta}{4},
\end{equation}
then $\parens{I_\beta \ast \abs{u}^p} \abs{u}^{p - 2} u \in L^2$
for every $u \in X$ (see Lemma \ref{prelim:lem:conditionForIntegrability}) and we define the rigorous version of \eqref{intro:eqn:HartreeDelta} as
\begin{equation}
\label{intro:eqn:RigorousHartreeDelta}
i \partial_t \psi
=
\op \psi
-
\parens{I_\beta \ast \abs{\psi}^p} \abs{\psi}^{p - 2} \psi.
\end{equation}
We then say that $u \in \Dom \parens{\op}$ is a \emph{bound state} of \eqref{intro:eqn:HartreeDelta} with frequency $\omega \in \real$ when
\begin{equation}
\label{intro:eqn:boundStateHartreeDelta}
\op u + \omega u - \parens{I_\beta \ast \abs{u}^p} \abs{u}^{p - 2} u
=
0
\quad \text{in} \quad L^2,
\end{equation}
in which case $\psi \parens{x, t} := u \parens{x} e^{i \omega t}$ satisfies \eqref{intro:eqn:RigorousHartreeDelta} in $L^2$.

The operator $\op$ is canonically associated with the quadratic form $Q \colon X \to \real$ given by
\[
Q \parens{u}
:=
\norm{\phi}_{D^{1, 2}}^2
+
\lambda \parens*{\norm{\phi}_{L^2}^2 - \norm{u}_{L^2}^2}
+
\parens*{\alpha + \frac{\sqrt{\lambda}}{4 \pi}} \abs{q}^2
\]
for every $u := \phi + q \G_\lambda \in X$, where
\[
\Dom \parens{\op}
\subset
X
:=
\set*{
	\phi + q \G_\lambda:
	\parens{\phi, q, \lambda}
	\in
	H^1 \times \complex \times \ooi{0, \infty}
}
\subset
L^2.
\]
Notice that
\[
\inf_{u \in X \setminus \set{0}}
\frac{Q \parens{u}}{\norm{u}_{L^2}^2}
=
\begin{cases}
- \parens{4 \pi \alpha}^2, 	&\text{if} ~ \alpha < 0;
\\
0,								&\text{if} ~ \alpha \geq 0.
\end{cases}
\]

An application of the HLS Inequality shows that if
\begin{equation}
\label{intro:eqn:conditionWellDefinedEnergy}
\frac{3 + \beta}{3}
\leq
p
<
\frac{3 + \beta}{2},
\end{equation}
then
\[
N \parens{u}
:=
\int \parens*{\parens{I_\beta \ast \abs{u}^p} \abs{u}^p}
<
\infty
\]
for every $u \in X$. Under this condition, we have a well-defined \emph{energy functional} $E \colon X \to \real$ given by
\[
E \parens{u}
:=
\frac{1}{2} Q \parens{u}
-
\frac{1}{2 p} N \parens{u}.
\]
Notice that $E$ extends the more usual energy functional
$E^0 \colon H^1 \to \real$
given by
\[
E^0 \parens{\phi}
=
\frac{1}{2} \norm{\phi}_{D^{1,2}}^2
-
\frac{1}{2 p} N \parens{\phi}.
\]

Under condition \eqref{intro:eqn:conditionForChoquardL2}, critical points of $E|_{X_\mu}$ are actually bound states of \eqref{intro:eqn:RigorousHartreeDelta}, where
\[X_\mu := \set*{v \in X : \norm{v}_{L^2}^2 = \mu}.\]
Indeed, if $u := \phi + q \G_\lambda \in X_\mu$ is a critical point of $E|_{X_\mu}$ and $\omega$ is its Lagrange multiplier, that is, $E' \parens{u} = - \omega \angles{\cdot, u}_{L^2}$, then it follows from $H^2$-elliptic regularity that
$\phi \in H^2$. We can then verify that the boundary condition is satisfied by expanding
\[
\angles*{
	-
	\Delta \phi
	-
	q \lambda \G_\lambda
	+
	\omega u
	-
	\parens*{I_\beta \ast \abs{u}^p} \abs{u}^{p - 2} u,
	\G_\lambda
}_{L^2}
=
0.
\]

In this context, we say that $u$ is a \emph{ground state} of \eqref{intro:eqn:RigorousHartreeDelta} at mass $\mu$ when $u$ is a minimizer of $E|_{X_\mu}$ and we denote the infima of the energy under the mass constraint as
\[
\E \parens{\mu} := \inf_{v \in X_\mu} E \parens{v}
\quad \quad \text{and} \quad \quad
\E^0 \parens{\mu} := \inf_{\phi \in H^1_\mu} E^0 \parens{\phi},
\]
where
$
H^1_\mu := \set*{\phi \in H^1 : \norm{\phi}_{L^2}^2 = \mu}
$.

\subsection{Main results}

Before stating our first result, we recall the Pohožaev identity for bound states of \eqref{intro:eqn:Hartree} which follows from \cite[Proposition 3.1]{morozGroundstatesNonlinearChoquard2013}.

\begin{prop}
\label{intro:prop:PohozaevHartree}
If $\phi \in H^1 \cap L^{\frac{6 p}{3 + \beta}}$ is a bound state of \eqref{intro:eqn:Hartree} and
$\phi \in H^2_{loc} \cap W^{1, \frac{6p}{3 + \beta}}$, then the following Pohožaev identity holds:
\begin{equation}
\label{intro:eqn:PohozaevHartree}
\frac{1}{2} \norm{\phi}_{D^{1, 2}}^2
+
\frac{3}{2} \omega \norm{\phi}_{L^2}^2
=
\frac{3 + \beta}{2 p} N \parens{\phi}.
\end{equation}
\end{prop}

Our first result is that bound states of \eqref{intro:eqn:RigorousHartreeDelta} satisfy a Pohožaev identity.

\begin{prop}
\label{intro:prop:PohozaevHartreeDelta}
If condition \eqref{intro:eqn:conditionForChoquardL2} is satisfied and $u := \phi + q \G_\lambda$ is a bound state of \eqref{intro:eqn:RigorousHartreeDelta} with frequency
$\omega$, then the following Pohožaev identity holds:
\begin{equation}
\label{intro:eqn:PohozaevHartreeDelta}
\frac{1}{2} Q \parens{u}
+
\frac{1}{2} \alpha \abs{q}^2
+
\frac{3}{2} \omega \norm{u}_{L^2}^2
=
\frac{3 + \beta}{2 p} N \parens{u}.
\end{equation}
\end{prop}

Notice that when $\alpha = 0$, \eqref{intro:eqn:PohozaevHartreeDelta} reads as \eqref{intro:eqn:PohozaevHartree}. Furthemore,
$\parens{5 + 2 \beta} / 4 < \parens{3 + \beta} / 2$, so \eqref{intro:eqn:conditionForChoquardL2} implies \eqref{intro:eqn:conditionWellDefinedEnergy} and the integrals in \eqref{intro:eqn:PohozaevHartreeDelta} do not diverge. The theorem is proved by means of a weaker form of the divergence theorem (see Lemma \ref{Pohozaev:lem:pseudo_Greens_first_identity}), with which we can argue as in \cite[Proof of Proposition 3.1]{morozGroundstatesNonlinearChoquard2013}. We can then use dominated convergence, the properties of the function
$\G_\lambda$ and the fact that $-\Delta + \lambda$ is self-adjoint to conclude.

A first consequence of the Pohožaev identity is the following result of nonexistence of ground states.
\begin{thm}
\label{intro:thm:nonexistence_of_ground_states}
If $p = \parens{3 + \beta} / 3$ and $\alpha \geq 0$, then \eqref{intro:eqn:RigorousHartreeDelta} does not admit a ground state at any mass $\mu > 0$.
\end{thm}

We prove Theorem \ref{intro:thm:nonexistence_of_ground_states} by following the general argument in \cite[Proof of Theorem 1.1]{yeMassMinimizersConcentration2016} and taking into account the Pohožaev identity \eqref{intro:eqn:PohozaevHartreeDelta}. Notice that the hypothesis $\alpha \geq 0$ is essential in order to assure that $Q$ is positive-definite and argue as such.

Another consequence of the Pohožaev identity \eqref{intro:eqn:PohozaevHartreeDelta} is the following result about the form of bound states of \eqref{intro:eqn:RigorousHartreeDelta}.

\begin{prop}
\label{intro:prop:bound_states}
Suppose that $p = \parens{3 + \beta} / 3$, $\alpha \geq 0$ and
$u := \phi + q \G_\lambda$ is a bound state of \eqref{intro:eqn:RigorousHartreeDelta}. If either $q = 0$ or
$\phi \equiv 0$, then $u \equiv 0$.
\end{prop}

We remark that analogs to the previous results also hold for NLSE in
$\real^3$ with a delta potential. Our last result is concerned with the existence of ground states of \eqref{intro:eqn:RigorousHartreeDelta}.
\begin{thm}
\label{intro:thm:existence_of_ground_states}
Suppose that
\begin{equation}
\label{intro:eqn:conditionForGroundState}
\frac{3 + \beta}{3}
<
p
<
\min \parens*{
	\frac{5 + \beta}{3}, \frac{5 + 2 \beta}{4}
}.
\end{equation}
Given $\mu > 0$, \eqref{intro:eqn:RigorousHartreeDelta} admits a ground state at mass $\mu$ and any such function is in $X_\mu \setminus H^1_\mu$.
\end{thm}

In particular, the conclusion of the theorem holds in the physically-relevant case $p = \beta = 2$. We also highlight that even when $\alpha = 0$, the ground state of \eqref{intro:eqn:RigorousHartreeDelta} whose existence is guaranteed by the theorem is different (has strictly smaller energy) than the usual ground state of \eqref{intro:eqn:Hartree} (which lies in $H^1_\mu$).

Let us dissect \eqref{intro:eqn:conditionForGroundState}. Our arguments will use the fact that ground states are bound states, so we want \eqref{intro:eqn:conditionForChoquardL2} to hold, hence the term
$\parens{5 + 2 \beta} / 4$. We also use the existence of ground states of \eqref{intro:eqn:Hartree} at any mass $\mu > 0$ and it was proved in \cite{yeMassMinimizersConcentration2016} that these ground states exist if, and only if,
$\parens{3 + \beta} / 3 < p < \parens{5 + \beta} / 3$,
hence the term $\parens{5 + \beta} / 3$.

The proof of the theorem follows the energy comparison argument used to prove \cite[Theorem 1.5 (i)]{adamiGroundStatesPlanar2022} and \cite[Theorem I.2 (i)]{adamiExistenceStructureRobustness2022}. Similarly to the GN-type inequality \cite[Proposition II.1]{adamiExistenceStructureRobustness2022} (see Proposition \ref{groundStates:prop:AdamiProp2.1}), we have to develop an interpolation inequality for our context in Proposition \ref{groundStates:prop:generalizationOfYe}. Notice that the proof of \cite[Theorem I.2 (i)]{adamiExistenceStructureRobustness2022} depends indirectly on \cite[Proposition III.3]{adamiExistenceStructureRobustness2022}, which mentions the radiality of ground states of the NLSE. Even though it seems to be unknown whether every ground state of the Hartree-type equation \eqref{intro:eqn:Hartree} is radial, this property is not used in the proof of \cite[Proposition III.2]{adamiExistenceStructureRobustness2022} (see the comment after \cite[Theorem 1.1]{georgievStandingWavesGlobal2022}), so an analogous result (that is, Lemma \ref{groundStates:lem:AdamiProp3.2}) also holds in the present context, hence the applicability of Adami, Boni, Carlone \& Tentarelli's argument.

\subsection*{Organization of the text}

We develop a few preliminaries in Section \ref{prelim}. The Pohožaev identity and its consequences are proved in Section \ref{Pohozaev}. To finish, Theorem \ref{intro:thm:existence_of_ground_states} is proved in Section \ref{groundStates}.

\subsection*{Notation and abbreviations}

Brackets are exclusively employed to enclose the argument of (multi-)linear functions. We let $D^{1, 2}$ denote the Hilbert space completion of $C_c^\infty$ with respect to
$
\angles{u, v}_{D^{1, 2}}
:=
\int \parens{\nabla u \cdot \nabla \overline{v}}
$. Throughout the text, we employ the following abbreviations:
\begin{center}
\begin{tabular}{c | c}
BL		&Brézis--Lieb;
\\
GN		&Gagliardo--Nirenberg;
\\
LHS		&Left-Hand Side;
\\
HLS		&Hardy--Littlewood--Sobolev;
\\
NLSE	&Nonlinear Schrödinger Equation;
\\
RHS		&Right-Hand Side.
\end{tabular}
\end{center}

\subsection*{Acknowledgements}

The author thanks Jaime Angulo Pava for informing about \cite{adamiExistenceStructureRobustness2022} and conversations about the topic. The author also thanks Gaetano Siciliano for useful suggestions. This study was financed in part by the Coordenação de Aperfeiçoamento de Pessoal de Nível Superior - Brasil (CAPES) - Finance Code 001.

\section{Preliminaries}
\label{prelim}

We want to recall a few computational results related to the functions $\G_\lambda$.

\begin{lem}
\label{prelim:lem:Glambda}
Suppose that $\lambda \in \ooi{0, \infty}$ and
$2 \leq r < 3$. We have
\[
\norm{\G_\lambda}_{L^2}^2 = \frac{1}{8 \pi \sqrt{\lambda}}
\quad \quad \text{and} \quad \quad
\norm{\G_\lambda}_{L^r}^r
=
\frac{1}{\lambda^{\frac{3 - r}{2}}} \norm{\G_1}_{L^r}^r.
\]
Given $\nu \in \ooi{0, \infty} \setminus \set{\lambda}$, the following equalities hold:
\[
\norm{\G_\lambda - \G_\nu}_{L^2}^2
=
\frac{1}{8 \pi}
\parens*{
	\frac{1}{\sqrt{\lambda}}
	+
	\frac{1}{\sqrt{\nu}}
	-
	\frac{4}{\sqrt{\lambda} + \sqrt{\nu}}
}
\]
and
\[
\norm{\G_\lambda - \G_\nu}_{D^{1, 2}}^2
=
\frac{1}{8 \pi}
\cdot
\frac{
	3 \lambda \sqrt{\nu} - 3 \nu \sqrt{\lambda}
	+
	\nu \sqrt{\nu} - \lambda \sqrt{\lambda}
}{\nu - \lambda}.
\]
\end{lem}

Let us deduce a sufficient condition on $p$ to assure that
$\parens{I_\beta \ast \abs{u}^p} \abs{u}^{p - 2} u \in L^2$ for every $u \in X$.
\begin{lem}
\label{prelim:lem:conditionForIntegrability}
If $u \in L^r$ for $r \in \coi{2, 3}$ and
\begin{equation}
\label{appendix:eqn:theCondition}
\frac{3 + \beta}{3}
\leq
p
<
\frac{5 + 2 \beta}{4},
\end{equation}
then
$\parens{I_\beta \ast \abs{u}^p} \abs{u}^{p - 2} u \in L^2$.
\end{lem}
\begin{proof}
Suppose that $1 < p < 3$ and
\begin{equation}
\label{appendix:eqn:aux:1}
\max \parens*{1, \frac{2}{p}}
\leq
\delta
<
\frac{3}{p},
\end{equation}
so that $\abs{u}^p \in L^\delta$. It follows from the HLS Inequality that
$I_\beta \ast \abs{u}^p \in L^\gamma$, where
\[\gamma := \frac{3 \delta}{3 - \beta \delta}.\]
Now, suppose that $\eps, \zeta \in \ooi{1, \infty}$ are conjugate exponents such that $\parens{I_\beta \ast \abs{u}^p}^2 \in L^\eps$ and $\abs{u}^{2 \parens{p - 1}} \in L^\zeta$. It follows from Hölder's inequality that
\[
\int \parens*{
	\parens{I_\beta \ast \abs{u}^p}^2 \abs{u}^{2 \parens{p - 1}}
}
\leq
\norm*{I_\beta \ast \abs{u}^p}_{L^{2 \eps}}^2
\norm{u}^{2 \parens{p - 1}}_{L^{2 \parens{p - 1} \zeta}}.
\]
We know that $I_\beta \ast \abs{u}^p \in L^\gamma$, so we can take
$\eps = \gamma / 2$, which implies
\[
\zeta
=
\frac{\eps}{\eps - 1}
=
\frac{\gamma}{\gamma - 2}
=
\frac{3 \delta}{\parens{3 + 2 \beta} \delta - 6}.
\]
As $u \in L^{2 \parens{p - 1} \zeta}$ and we only know that
$u \in L^r$ for $r \in \coi{2, 3}$, we deduce that
\begin{equation}
\label{appendix:eqn:aux:2}
2
\leq 
2 \parens{p - 1} \frac{3 \delta}{\parens{3 + 2 \beta} \delta - 6}
<
3.
\end{equation}
It then follows from elementary manipulations of inequalities that \eqref{appendix:eqn:aux:1}, \eqref{appendix:eqn:aux:2} may only hold simultaneously when \eqref{appendix:eqn:theCondition} holds.
\end{proof}

A specific value of $\eps > 0$ will be useful for our computations.

\begin{rmk}
\label{prelim:rmk:the_decomposition}
With $\eps := 1 / \parens{64 \pi ^2}$, we have
\[
Q \parens{u}
=
\frac{1}{2} \norm{\phi}_{D^{1, 2}}^2
+
\frac{\abs{q}^4}{128 \pi^2 \norm{u}_{L^2}^2} \parens*{
	1 + \frac{\norm{\phi}_{L^2}^2}{\norm{u}_{L^2}^2}
}
+
\frac{\alpha}{2} \abs{q}^2
\]
for every
\[
u
:=
\phi + q \G_{\eps \abs{q}^4 / \norm{u}_{L^2}^4} \in X \setminus H^1.
\]
\end{rmk}

\section{The Pohožaev identity and its consequences}
\label{Pohozaev}

We need the following weaker form of the divergence theorem to prove Proposition \ref{intro:prop:PohozaevHartreeDelta}.

\begin{lem}
\label{Pohozaev:lem:pseudo_Greens_first_identity}
If $\eta \in C_c^\infty$, $2 \leq r < 3$ and
$u \in \Dom \parens{\op}$, then
\[
\int \parens*{
	D_\eta \cdot \nabla \parens*{\abs{u}^r}
}
+
\int \parens*{
	\abs{u}^r \divergence D_\eta
}
=
0,
\]
where $D_\eta \parens{x} := \eta \parens{x} x \in \real^3$.
\end{lem}
\begin{proof}
Suppose that
$u := \phi + q \G_\lambda \in \Dom \parens{\op}$.
It follows from the product rule that
\begin{equation}
\label{Pohozaev:eqn:aux:8}
D_\eta \cdot \nabla \parens*{\abs{u}^r}
+
\abs{u}^r \divergence D_\eta
=
\divergence \parens*{\abs{u}^r D_\eta}.
\end{equation}
Let us show that the terms on both the LHS and the RHS are integrable.

We claim that
\begin{equation}
\label{Pohozaev:eqn:aux:5}
\abs{u}^r D_\eta \in W^{1, 1} \parens{\real^3, \real^3}.
\end{equation}
It is clear that $\abs{u}^r D_\eta \in L^1 \parens{\real^3, \real^3}$ because $\eta \in C_c^\infty$ and $u \in L^r$. Let us prove that
$
\parens{\abs{u}^r D_\eta}'
\in
L^1 \parens{\real^3, \real^{3 \times 3}}
$.
It follows from the product rule that
\begin{multline*}
\parens*{\abs{u}^r D_\eta}' \parens{x} \brackets{y}
=
\parens*{
	r \abs*{u \parens{x}}^{r - 2}
	\Re \brackets*{
		u \parens{x}
		\nabla \overline{u} \parens{x} \cdot y
	}
	\eta \parens{x}
	+
	\abs*{u \parens{x}}^r
	\nabla \eta \parens{x} \cdot y
}
x
+
\\
+
\abs*{u \parens{x}}^r \eta \parens{x} y,
\end{multline*}
so
\[
\abs*{\parens*{\abs{u}^r D_\eta}' \parens{x}}
\leq
r
\abs{x}
\abs*{u \parens{x}}^{r - 1}
\abs*{\nabla u \parens{x}}
\eta \parens{x}
+
\abs*{u \parens{x}}^r 
\parens*{
	\abs*{\nabla \eta \parens{x}} \abs{x}
	+
	\eta \parens{x}
}.
\]
As
$
\abs{u}^r \parens{\abs{\cdot} \abs{\nabla \eta} + \eta}
\in
L^1
$,
because $u \in L^r$ and $\eta \in C_c^\infty$, we only have to prove that $\abs{\cdot} \abs{u}^{r - 1} \abs{\nabla u} \eta \in L^1$
to conclude. Due to the Hölder inequality, it suffices to show that
$\abs{\cdot} \abs{\nabla u} \eta \in L^r$. On one hand,
$\abs{\cdot} \abs{\nabla \phi} \eta \in L^r$, because
$\phi \in H^2 \hookrightarrow W^{1, r}$ and
$\eta \in C_c^\infty$. On the other hand, we have
$\abs{\cdot} \abs{\nabla \G_\lambda} \eta \in L^r$, because
\[
\abs{\cdot} \abs{\nabla \G_\lambda}
=
\parens*{
	\sqrt{\lambda} \abs{\cdot}
	+
	1
}
\G_\lambda
\in
L^r.
\]

The fact that
$D_\eta \cdot \nabla \parens*{\abs{u}^r} \in L^1$
follows from the previous paragraph, while it is clear that
$\abs{u}^r \divergence D_\eta \in L^1$. By integrating both sides in \eqref{Pohozaev:eqn:aux:8}, we obtain
\[
\int \parens*{
	D_\eta \cdot \nabla \parens*{\abs{u}^r}
}
+
\int \parens*{
	\abs{u}^r \divergence D_\eta
}
=
\int \divergence \parens*{\abs{u}^r D_\eta}.
\]
Due to \eqref{Pohozaev:eqn:aux:5}, the result follows from the divergence theorem.
\end{proof}

We proceed to the proof of the Pohožaev identity.

\begin{proof}[Proof of Proposition \ref{intro:prop:PohozaevHartreeDelta}]

Let $\eta \in C_c^\infty \parens{\real^3, \cci{0, 1}}$ be such that
$\eta \parens{x} = 1$ if $\abs{x} < 1$. Associate each pair
$\parens{u, \mu} \in X \times \ooi{0, \infty}$
to the function given for a.e. $x \in \real^3$ by
\[
v_\mu \brackets{u} \parens{x}
:=
\frac{1}{\mu} D_\eta \parens{\mu x} \cdot \nabla u \parens{x},
\]
with $D_\eta$ as defined in Lemma \ref{Pohozaev:lem:pseudo_Greens_first_identity}.

Suppose that
$u := \phi + q \G_\lambda$ is a bound state of \eqref{intro:eqn:RigorousHartreeDelta} with frequency $\omega$. By integrating \eqref{intro:eqn:boundStateHartreeDelta} against
$
v_\mu \brackets{\overline{u}}
=
v_\mu \brackets{~\overline{\phi}~}
+
\overline{q} v_\mu \brackets{\G_\lambda}
$,
we obtain
\begin{multline*}
\underbrace{
	\int \parens*{
		- v_\mu \brackets{~\overline{\phi}~} \Delta \phi
	}
}_{\parens{\ast}_\mu}
+
\underbrace{\parens*{
	- \lambda \abs{q}^2
	\int \parens*{\G_\lambda v_\mu \brackets{\G_\lambda}}
}}_{\parens{\ast \ast}_\mu}
+
\underbrace{
	\omega \int \parens*{u v_\mu \brackets{\overline{u}}}
}_{\parens{\ast \ast \ast}_\mu}
+
\\
+
\underbrace{
	\parens*{
		- \int \parens*{
			\parens*{I_\beta \ast \abs{u}^p} \abs{u}^{p - 2} u
			v_\mu \brackets{\overline{u}}
		}
	}
}_{\parens{\ast \ast \ast \ast}_\mu}
+
\underbrace{
	\int \parens*{
		-
		\overline{q} v_\mu \brackets{\G_\lambda} \Delta \phi
		-
		q
		\lambda
		\G_\lambda
		v_\mu
		\brackets{~\overline{\phi}~}
	}
}_{\parens{\ast \ast \ast \ast \ast}_\mu}
=
0.
\end{multline*}

As $\phi \in H^2$, dominated convergence shows that
$\parens{\ast}_\mu \to - \norm{\phi}_{D^{1, 2}}^2 / 2$
as $\mu \to 0$. In view of Lemma \ref{Pohozaev:lem:pseudo_Greens_first_identity}, dominated convergence also shows that
\[
\parens{\ast \ast}_\mu
\to
\frac{3}{2} \lambda \abs{q}^2 \norm{\G_\lambda}_{L^2}^2
=
\frac{3 \sqrt{\lambda}}{16 \pi} \abs{q}^2;
\quad \quad
\parens{\ast \ast \ast}_\mu
\to
- \frac{3}{2} \omega \norm{u}_{L^2}^2
\]
and
$
\parens{\ast \ast \ast \ast}_\mu
\to
\parens{3 + \beta} N \parens{u} / \parens{2 p}
$
as $\mu \to 0$. (for details, see \cite[Proof of Proposition 3.1]{morozGroundstatesNonlinearChoquard2013})

Once again, dominated convergence implies
\[
\parens{\ast \ast \ast \ast \ast}_\mu
\to
\underbrace{
\overline{q}
	\int
		- \Delta \phi \parens{x}
		\parens*{x \cdot \nabla \G_\lambda \parens{x}}
	\dif x
}_{\parens{\dagger}}
-
\underbrace{
	q \lambda
	\int
		\G_\lambda \parens{x}
		\parens*{x \cdot \nabla \overline{\phi} \parens{x}}
	\dif x
}_{\parens{\dagger \dagger}}
\]
as $\mu \to 0$. As
\[
x \cdot \nabla \G_\lambda \parens{x}
=
-
\parens*{
	\sqrt{\lambda} \abs{x}
	+
	1
}
\G_\lambda \parens{x},
\]
we obtain
\begin{align*}
\parens{\dagger}
&=
\frac{\sqrt{\lambda}}{4 \pi} \overline{q}
\int \Delta \phi \parens{x} e^{- \sqrt{\lambda} \abs{x}} \dif x
-
\overline{q}
\int - \Delta \phi \parens{x} \G_\lambda \parens{x} \dif x;
\\
&=
\frac{\sqrt{\lambda}}{4 \pi} \overline{q}
\int \Delta \phi \parens{x} e^{- \sqrt{\lambda} \abs{x}} \dif x
-
\overline{q}
\int
	\phi \parens{x} \parens{- \Delta + \lambda} \G_\lambda \parens{x}
\dif x
+
\lambda \overline{q} \angles{\phi, \G_\lambda}_{L^2};
\\
&=
\frac{\sqrt{\lambda}}{4 \pi} \overline{q}
\int \Delta \phi \parens{x} e^{- \sqrt{\lambda} \abs{x}} \dif x
-
\phi \parens{0} \overline{q}
+
\lambda \overline{q} \angles{\phi, \G_\lambda}_{L^2}.
\end{align*}
Integrating by parts, we obtain
\[
\parens{\dagger}
-
\overline{\parens{\dagger \dagger}}
=
-
\phi \parens{0} \overline{q}
+
\lambda \overline{q} \angles{\phi, \G_\lambda}_{L^2},
\]
and thus
\[
\Re \brackets*{
	\parens{\dagger}
	-
	\parens{\dagger \dagger}
}
=
-
\phi \parens{0} \overline{q}
+
\lambda \angles*{ \Re \brackets{\overline{q} \phi}, \G_\lambda}_{L^2}.
\]
\end{proof}

Let us highlight an identity which follows from the Pohožaev identity \eqref{intro:eqn:PohozaevHartreeDelta} when
$p = \parens{3 + \beta} / 3$.

\begin{rmk}
Suppose that $p = \parens{3 + \beta} / 3$ and
$u := \phi + q \G_\lambda$ is a bound state of \eqref{intro:eqn:RigorousHartreeDelta} with frequency $\omega$. In view of the \emph{Nehari identity}
$Q \parens{u} + \omega \norm{u}_{L^2}^2 = N \parens{u}$
and the Pohožaev identity \eqref{intro:eqn:PohozaevHartreeDelta}, we deduce that
\begin{equation}
\label{Pohozaev:eqn:another_identity}
Q \parens{u} = \frac{1}{2} \alpha \abs{q}^2.
\end{equation}
\end{rmk}

The previous identity will play a major role in the proofs of the consequences of Proposition \ref{intro:prop:PohozaevHartreeDelta}.

\begin{proof}
[Proof of Theorem \ref{intro:thm:nonexistence_of_ground_states}]
It follows from \cite{liebSharpConstantsHardylittlewoodsobolev1983} that
\begin{equation}
\label{Pohozaev:eqn:HLS}
N \parens{u}
\leq
\parens*{
	\frac{\norm{u}_{L^2}^2}{\norm{Q}_{L^2}^2}
}^{\frac{3 + \beta}{3}}
\quad \text{for every} \quad
u \in L^2
\end{equation}
and equality holds if, and only if, $u = Q$, where $Q \colon \real^3 \to \real$ is given by
\[
Q \parens{x}
:=
C \frac{\eta}{\parens{\eta^2 + \abs{x - a}^2}^{\frac{3}{2}}}
\]
and $a \in \real^3$, $C, \eta > 0$ are parameters. As $\alpha \geq 0$, it follows from \eqref{Pohozaev:eqn:HLS} that
\[
\E \parens{\mu}
\geq
-
\frac{3}{2 \parens{3 + \beta}}
\parens*{\frac{\mu}{\norm{Q}_{L^2}^2}}^{\frac{3 + \beta}{3}}.
\]

We claim that
\[
\E \parens{\mu}
=
-
\frac{3}{2 \parens{3 + \beta}}
\parens*{\frac{\mu}{\norm{Q}_{L^2}^2}}^{\frac{3 + \beta}{3}}.
\]
Indeed, associate each $t > 0$ to the function
$Q_t \colon \real^3 \to \real$ defined as
\[
Q_t \parens{x}
=
\frac{\sqrt{\mu} t^{3 / 2}}{\norm{Q}_{L^2}} Q \parens{t x},
\]
so that $\norm{Q_t}_{L^2}^2 = \mu$. The result then follows from the fact that
\[
E \parens{Q_t}
\to
- \frac{3}{2 \parens{3 + \beta}}
\parens*{\frac{\mu}{\norm{Q}_{L^2}^2}}^{\frac{3 + \beta}{3}}
\]
as $t \to 0$.

Suppose that $u := \phi + q \G_\lambda$ is a ground state of \eqref{intro:eqn:RigorousHartreeDelta} at mass $\mu$. As
$\alpha \geq 0$, we have $Q \parens{u} \geq 0$. Therefore,
\[
0
\leq
\frac{1}{2} Q \parens{u}
=
\frac{3}{2 \parens{3 + \beta}}
\parens*{
	N \parens{u}
	-
	\parens*{\frac{\mu}{\norm{Q}_{L^2}^2}}^{\frac{3 + \beta}{3}}
}.
\]
In view of \eqref{Pohozaev:eqn:HLS}, we deduce that
$Q \parens{u} = 0$. Due to \eqref{Pohozaev:eqn:another_identity}, we deduce that either $\alpha = 0$ or $q = 0$. If $q = 0$, then
$Q \parens{u} = \norm{u}_{D^{1, 2}}^2 = 0$, so $u \equiv 0$. We claim that the case $q \neq 0$ is impossible. Indeed, suppose that
$q \neq 0$ and $\alpha = 0$. Let $\eps = 1 / \parens{64 \pi^2}$ and
$\phi_\eps \in H^2$ be such that
$u = \phi_\eps + q \G_{\eps \abs{q}^4 / \norm{u}_{L^2}^4}$.
In view of Remark \ref{prelim:rmk:the_decomposition}, 
$\phi_\eps \equiv 0$ and $q = 0$, which contradicts the fact that $q \neq 0$. Contradiction.
\end{proof}

We proceed to the proof of the second consequence of Proposition \ref{intro:prop:PohozaevHartreeDelta}.

\begin{proof}[Proof of Proposition \ref{intro:prop:bound_states}]
If $q = 0$, then it follows from \eqref{Pohozaev:eqn:another_identity} that $Q \parens{u} = \norm{u}_{D^{1, 2}}^2 = 0$, so $u \equiv 0$. If $\phi \equiv 0$, then it follows from \eqref{Pohozaev:eqn:another_identity} that
\[
\frac{\sqrt{\lambda}}{8 \pi} \abs{q}^2
=
- \frac{1}{2} \alpha \abs{q}^2,
\]
so $q = 0$.
\end{proof}

\section{Existence of ground states}
\label{groundStates}

We need to develop a few preliminaries before proving Theorem \ref{intro:thm:existence_of_ground_states}. Some of them are analogous to results in \cite[Section III]{adamiExistenceStructureRobustness2022} (see also \cite{adamiGroundStatesPlanar2022}) and are proved similarly, but we include their proof here for the convenience of the reader. Our first goal is to recall two inequalities. The first one is a recent generalization of the GN Inequality.

\begin{prop}[{\cite[Proposition II.1]{adamiExistenceStructureRobustness2022}}]
\label{groundStates:prop:AdamiProp2.1}
Suppose that $2 < r < 3$. There exists $K > 0$ (which depends on $r$) such that
\[
\norm{u}_{L^r}^r
\leq
K
\parens*{
	\norm{\phi}_{D^{1, 2}}^{\frac{3 \parens{r - 2}}{2}}
	\norm{\phi}_{L^2}^{\frac{6 - r}{2}}
	+
	\frac{\abs{q}^r}{\lambda^{\frac{3 - r}{2}}}
}
\]
for every $u := \phi + q \G_\lambda \in X$.
We can also associate each $\eps > 0$ to a $K_\eps > 0$ (which depends on $r$) for which
\[
\norm{u}_{L^r}^r
\leq
K_\eps
\parens*{
	\norm{\phi}_{D^{1, 2}}^{\frac{3 \parens{r - 2}}{2}}
	\norm{u}_{L^2}^{\frac{6 - r}{2}}
	+
	\abs{q}^{3 \parens{r - 2}} \norm{u}_{L^2}^{2 \parens{3 - r}}
}
\]
for every
\[
u
:=
\phi + q \G_{\eps \abs{q}^4 / \norm{u}_{L^2}^4}
\in X \setminus H^1.
\]
\end{prop}

The second one is a particular case of the interpolation inequality \cite[(1.6)]{yeMassMinimizersConcentration2016}.

\begin{lem}
\label{groundStates:lem:Ye1.6}
If $\parens{3 + \beta} / 3 < p < 3 + \beta$, then there exists
$K > 0$ (which depends on $p$) such that
\[
N \parens{\phi}
\leq
K
\norm{\phi}_{D^{1, 2}}^{3 p - \parens{3 + \beta}}
\norm{\phi}_{L^2}^{3 + \beta - p}
\]
for every $\phi \in H^1$.
\end{lem}

Let us obtain an analog of the previous inequality for functions in $X$.

\begin{prop}
\label{groundStates:prop:generalizationOfYe}
If $\parens{3 + \beta} / 3 < p < \parens{3 + \beta} / 2$,
then there exists $K > 0$ (which depends on $p$) such that
\[
N \parens{u}
\leq
K
\parens*{
	\norm{\phi}_{D^{1, 2}}^{3 p - \parens{3 + \beta}}
	\norm{\phi}_{L^2}^{3 + \beta - p}
	+
	\norm{\phi}_{D^{1, 2}}^{\frac{3 p - \parens{3 + \beta}}{2}}
	\norm{\phi}_{L^2}^{\frac{3 + \beta - p}{2}}
	\frac{\abs{q}^p}{\lambda^{\frac{3 + \beta - 2 p}{4}}}
	+
	\frac{\abs{q}^{2 p}}{\lambda^{\frac{3 + \beta - 2 p}{2}}}
}
\]
and
\begin{multline*}
N \parens{u}
\leq
K
\left(
	\norm{\phi}_{D^{1, 2}}^{3 p - \parens{3 + \beta}}
	\parens*{
		\norm{u}_{L^2}^{3 + \beta - p}
		+
		\frac{\abs{q}^{3 + \beta - p}}{
			\lambda^{\frac{3 + \beta - p}{4}}
		}
	}
	+
\right.
\\
\left.
+
\norm{\phi}_{D^{1, 2}}^{\frac{3 p - \parens{3 + \beta}}{2}}
\parens*{
	\norm{u}_{L^2}^{\frac{3 + \beta - p}{2}}
	+
	\frac{\abs{q}^{\frac{3 + \beta - p}{2}}}{
		\lambda^{\frac{3 + \beta - p}{8}}
	}
}
\frac{\abs{q}^p}{\lambda^{\frac{3 + \beta - 2 p}{4}}}
+
\frac{\abs{q}^{2 p}}{\lambda^{\frac{3 + \beta - 2 p}{2}}}
\right)
\end{multline*}
for every $u := \phi + q \G_\lambda \in X$.
\end{prop}
\begin{proof}
It is clear that
\[
N \parens{u}
\leq
4^p
\parens*{
	N \parens{\phi}
	+
	2
	\abs{q}^p
	\int \parens*{\parens{I_\beta \ast \abs{\phi}^p} \G_\lambda^p}
	+
	\abs{q}^{2 p}
	N \parens{\G_\lambda}
}.
\]
In view of Lemma \ref{groundStates:lem:Ye1.6},
\[
N \parens{\phi}
\leq
K
\norm{\phi}_{D^{1, 2}}^{3 p - \parens{3 + \beta}}
\norm{\phi}_{L^2}^{3 + \beta - p}.
\]
It follows from the HLS Inequality that
\begin{multline*}
2
\abs{q}^p
\int \parens*{\parens{I_\beta \ast \abs{\phi}^p} \G_\lambda^p}
+
\abs{q}^{2 p}
N \parens{\G_\lambda}
\leq
\\
\leq
\underbrace{
	K
	\abs{q}^p
	\parens*{\norm{\phi}_{L^q}^q}^{\frac{3 + \beta}{6}}
	\parens*{\norm{\G_\lambda}_{L^q}^q}^{\frac{3 + \beta}{6}}
}_{\parens{\ast}}
+
\underbrace{
	K
	\abs{q}^{2 p}
	\parens*{\norm{\G_\lambda}_{L^q}^q}^{\frac{3 + \beta}{3}}
}_{\parens{\ast \ast}},
\end{multline*}
where $q := 6 p / \parens{3 + \beta}$. Consider the term
$\parens{\ast}$. In view of Proposition \ref{groundStates:prop:AdamiProp2.1},
\[
\parens{\ast}
\leq
K
\abs{q}^p
\norm{\phi}_{D^{1, 2}}^{\frac{3 p - \parens{3 + \beta}}{2}}
\norm{\phi}_{L^2}^{\frac{3 + \beta - p}{2}}
\parens*{\norm{\G_\lambda}_{L^q}^q}^{\frac{3 + \beta}{6}}.
\]
Considering Lemma \ref{prelim:lem:Glambda}, we have
\[
\parens{\ast}
\leq
\parens{\ast \ast \ast}
:=
K
\norm{\phi}_{D^{1, 2}}^{\frac{3 p - \parens{3 + \beta}}{2}}
\norm{\phi}_{L^2}^{\frac{3 + \beta - p}{2}}
\frac{\abs{q}^p}{\lambda^{\frac{3 + \beta - 2 p}{4}}}
\]
and
\[
\parens{\ast \ast}
=
K
\frac{\abs{q}^{2 p}}{\lambda^{\frac{3 + \beta - 2 p}{2}}}.
\]
As
\[
\norm{\phi}_{L^2}^{\frac{3 + \beta - p}{2}}
\leq
2^{\frac{3 + \beta - p}{2}}
\parens*{
	\norm{u}_{L^2}^{\frac{3 + \beta - p}{2}}
	+
	\frac{\abs{q}^{\frac{3 + \beta - p}{2}}}{
		\lambda^{\frac{3 + \beta - p}{8}}
	}
},
\]
we deduce that
\[
\parens{\ast \ast \ast}
\leq
K
\norm{\phi}_{D^{1, 2}}^{\frac{3 p - \parens{3 + \beta}}{2}}
\parens*{
	\norm{u}_{L^2}^{\frac{3 + \beta - p}{2}}
	+
	\frac{\abs{q}^{\frac{3 + \beta - p}{2}}}{
		\lambda^{\frac{3 + \beta - p}{8}}
	}
}
\frac{\abs{q}^p}{\lambda^{\frac{3 + \beta - 2 p}{4}}}.
\]
\end{proof}

We highlight the inequality obtained for a specific decomposition of functions in $X \setminus H^1$.

\begin{cor}
\label{groundStates:cor:interpolationInequalityDecomposition}
Suppose that $\parens{3 + \beta} / 3 < p < \parens{3 + \beta} / 2$. Given $\eps > 0$, we have
\begin{multline*}
N \parens{u}
\leq
K_\eps
\left(
	\norm{\phi}_{D^{1, 2}}^{3 p - \parens{3 + \beta}}
	\norm{u}_{L^2}^{3 + \beta - p}
	+
	\norm{\phi}_{D^{1, 2}}^{
		\frac{3 p - \parens{3 + \beta}}{2}
	}
	\norm{u}_{L^2}^{\frac{3 + \beta - p}{2}}
	\abs{q}^{3 p - \parens{3 + \beta}}
	+
\right.
\\
\left.
	+
	\norm{u}_{L^2}^{2 \parens{3 + \beta - 2 p}}
	\abs{q}^{2 \parens{3 p - \parens{3 + \beta}}}
\right)
\end{multline*}
for every
\[
u
:=
\phi + q \G_{\eps \abs{q}^4 / \norm{u}_{L^2}^4} \in X \setminus H^1,
\]
where
\[
K_\eps
:=
\max \parens*{
	1 + \frac{1}{\eps^{\frac{3 + \beta - p}{4}}},
	\frac{1}{\eps^{\frac{3 + \beta - 2 p}{4}}}
	\parens*{
		1 + \frac{1}{\eps^{\frac{3 + \beta - p}{8}}}
	},
	\frac{1}{\eps^{\frac{3 + \beta - 2 p}{2}}}
}
K
\]
and $K > 0$ is furnished by the previous proposition.
\end{cor}

Let us recall the necessary and sufficient condition on $p$ for the existence of ground states of \eqref{intro:eqn:Hartree} at any mass
$\mu > 0$ which follows from \cite[Theorem 1.1]{yeMassMinimizersConcentration2016}.

\begin{thm}
\label{groundStates:thm:YeThm1.1}
The Hartree-type equation \eqref{intro:eqn:Hartree} has a ground state at any mass $\mu > 0$ if, and only if,
\begin{equation}
\label{groundStates:eqn:conditionYe}
\frac{3 + \beta}{3} < p < \frac{5 + \beta}{3}.
\end{equation}
Furthermore, if \eqref{groundStates:eqn:conditionYe} holds, then
$- \infty < \E^0 \parens{\mu} < 0$ for every $\mu > 0$.
\end{thm}

In fact, ground states of \eqref{intro:eqn:Hartree} do not not vanish anywhere in $\real^3$.

\begin{lem}
\label{groundStates:lem:|S_mu|isPositive}
Suppose that $\parens{3 + \beta} / 3 < p < \parens{5 + \beta} / 3$ and $S$ denotes a ground state of the Hartree-type equation \eqref{intro:eqn:Hartree} at mass $\mu > 0$. We conclude that
$\abs{S}$ is of class $C^2$ and does not vanish in $\real^3$.
\end{lem}
\begin{proof}
It is clear that $\abs{S}$ is also a minimizer of
$E^0|_{H^1_\mu}$. As
\[
E^0 \parens{S}
=
\frac{1}{2} \norm{S}_{D^{1, 2}}^2 - \frac{1}{2 p} N \parens{S}
=
\E^0 \parens{\mu} < 0,
\]
it follows from the Nehari identity
$\norm{S}_{D^{1, 2}}^2 + \omega \norm{S}_{L^2}^2 = N \parens{S}$
that
\[
\omega
=
\frac{1}{\norm{S}_{L^2}^2} \parens*{N \parens{S} - \norm{S}_{D^{1, 2}}^2}
> 0.
\]
As such, the discussion in \cite[Section 3.3.1]{morozGuideChoquardEquation2017} shows that $\abs{S}$ is of class $C^2$. To finish, the inequality $\abs{S} > 0$ in $\real^3$ follows from the Strong Maximum Principle.
\end{proof}

The previous results imply $\E \parens{\mu} > - \infty$.

\begin{lem}
\label{groundStates:lem:prop3.1}
If
\[
\frac{3 + \beta}{3}
<
p
<
\min \parens*{\frac{5 + \beta}{3}, \frac{3 + \beta}{2}},
\]
then $\E \parens{\mu} > - \infty$ for every $\mu > 0$.
\end{lem}
\begin{proof}
In view of Theorem \ref{groundStates:thm:YeThm1.1}, we only have to prove that
$\inf_{u \in X_\mu \setminus H^1_\mu} \E \parens{u} > - \infty$. Consider a
\[
u
:=
\phi + q \G_{\eps \abs{q}^4 / \norm{u}_{L^2}^4}
\in X_\mu \setminus H^1_\mu,
\]
where $\eps := 1 / \parens{64 \pi^2}$. Due to Remark \ref{prelim:rmk:the_decomposition} and Corollary \ref{groundStates:cor:interpolationInequalityDecomposition},
\begin{multline*}
E \parens{u}
\geq
\frac{1}{2} \norm{\phi}_{D^{1, 2}}^2
+
\frac{\abs{q}^4}{128 \pi^2 \mu} \parens*{
	1 + \frac{\norm{\phi}_{L^2}^2}{\mu}
}
+
\frac{\alpha}{2} \abs{q}^2
-
\\
-
K_\eps
\norm{\phi}_{D^{1, 2}}^{3 p - \parens{3 + \beta}}
\mu^{\frac{3 + \beta - p}{2}}
-
\\
-
K_\eps
\norm{\phi}_{D^{1, 2}}^{\frac{3 p - \parens{3 + \beta}}{2}}
\mu^{\frac{3 + \beta - p}{4}}
\abs{q}^{3 p - \parens{3 + \beta}}
-
K_\eps
\mu^{3 + \beta - 2 p}
\abs{q}^{2 \parens{3 p - \parens{3 + \beta}}},
\end{multline*}
which is bounded below because $0 < 3 p - \parens{3 + \beta} < 2$.
\end{proof}

We proceed to a comparison of $\E \parens{\mu}$ with
$\E^0 \parens{\mu}$.

\begin{lem}
\label{groundStates:lem:AdamiProp3.2}
If
\[
\frac{3 + \beta}{3}
<
p
<
\min \parens*{
	\frac{5 + \beta}{3}, \frac{5 + 2 \beta}{4}
},
\]
then $\E \parens{\mu} < \E^0 \parens{\mu} < 0$ for every $\mu > 0$.
\end{lem}
\begin{proof}
In view of Theorem \ref{groundStates:thm:YeThm1.1}, we have
$\E^0 \parens{\mu} < 0$ and we can let $S \in X_\mu$ denote a ground state of the Hartree-type equation \eqref{intro:eqn:Hartree} at mass
$\mu$.

We claim that $S$ is not a minimizer of $E|_{X_\mu}$. Suppose otherwise. On one hand, $S \parens{0} = 0$ because $S$ is a bound state of \eqref{intro:eqn:HartreeDelta}. On the other hand, $\abs{S}$ is strictly positive due to Lemma \ref{groundStates:lem:|S_mu|isPositive}. Contradiction.

It follows from the previous paragraph that $E|_{X_\mu}$ does not have a minimizer in $H^1_\mu$, so $\E \parens{\mu} < \E^0 \parens{\mu}$ because $E|_{H^1} = E^0$.
\end{proof}

Let us show that, up to discarding a finite number of indices, minimizing sequences of $E|_{X_\mu}$ are bounded away from $H^1$.

\begin{lem}
\label{groundStates:lem:AdamiLem3.4}
Suppose that
\[
\frac{3 + \beta}{3}
<
p
<
\min \parens*{
	\frac{5 + \beta}{3}, \frac{5 + 2 \beta}{4}
},
\]
$\mu > 0$ and
$\parens{u_n := \phi_{\lambda, n} + q_n \G_\lambda}_{n \in \nat}$ is a minimizing sequence for $E|_{X_\mu}$. We conclude that there exist
$C > 0$ and $\bar{n} \in \nat$ such that $\abs{q_n} > C$ for every
$n \geq \bar{n}$.
\end{lem}
\begin{proof}
By contradiction, suppose that
$\parens{u_n := \phi_{\lambda, n} + q_n \G_\lambda}_{n \in \nat}$ is a minimizing sequence for $E|_{X_\mu}$ and $q_n \to 0$ as $n \to \infty$. In particular, $\parens{\phi_{\lambda, n}}_{n \in \nat}$ is bounded in $L^2$ because $\norm{\phi_{\lambda, n}}_{L^2}^2 \to \mu$ as $n \to \infty$.

We claim that $\parens{\phi_{\lambda, n}}_{n \in \nat}$ is bounded in $D^{1, 2}$. Indeed, $\parens{E \parens{u_n}}_{n \in \nat}$ is bounded, $0 < 3 p - \parens{3 + \beta} < 2$ and Proposition \ref{groundStates:prop:generalizationOfYe} implies
\begin{multline*}
E \parens{u_n}
\geq
\frac{1}{2} \norm{\phi_{\lambda, n}}_{D^{1, 2}}^2
+
\frac{\lambda}{2} \parens*{\norm{\phi_{\lambda, n}}_{L^2}^2 - \mu}
+
\frac{1}{2}
\parens*{\alpha + \frac{\sqrt{\lambda}}{4 \pi}}
\abs{q_n}^2
-
\\
-
K
\left(
	\norm{\phi_{\lambda, n}}_{D^{1, 2}}^{3 p - \parens{3 + \beta}}
	\norm{\phi_{\lambda, n}}_{L^2}^{3 + \beta - p}
	+
\right.
\\
\left.
	+
	\norm{\phi_{\lambda, n}}_{D^{1, 2}}^{
		\frac{3 p - \parens{3 + \beta}}{2}
	}
	\norm{\phi_{\lambda, n}}_{L^2}^{\frac{3 + \beta - p}{2}}
	\frac{\abs{q_n}^p}{\lambda^{\frac{3 + \beta - 2 p}{4}}}
	+
	\frac{\abs{q_n}^{2 p}}{\lambda^{\frac{3 + \beta - 2 p}{2}}}
\right)
\end{multline*}
for every $n \in \nat$, hence the result.

Let
\[
\xi_n
=
\frac{\sqrt{\mu}}{\norm{\phi_{\lambda, n}}_{L^2}}
\phi_{\lambda, n} \in H_\mu^1
\]
for every $n \in \nat$. As $u_n - \xi_n \to 0$ in $L^r$ for any
$r \in \coi{2, 3}$ as $n \to \infty$ and
$E \parens{u_n} \to \E \parens{\mu}$, we deduce that
$E^0 \parens{\xi_n} \to \E \parens{\mu}$, which contradicts Lemma \ref{groundStates:lem:AdamiProp3.2}.
\end{proof}

Consider the following compactness result which is proved with classical arguments.

\begin{lem}
\label{groundStates:lem:AdamiLem3.5}
If
\[
\frac{3 + \beta}{3}
<
p
<
\min \parens*{
	\frac{5 + \beta}{3}, \frac{5 + 2 \beta}{4}
},
\]
$\mu > 0$ and
$\parens{u_n := \phi_{\lambda, n} + q_n \G_\lambda}_{n \in \nat}$ is a minimizing sequence for $E|_{X_\mu}$, then
\begin{itemize}
\item
$\parens{\phi_{\lambda, n}}_{n \in \nat}$ and $\parens{q_n}_{n \in \nat}$ are respectively bounded in $H^1$ and $\complex$;
\item
there exists $u := \phi_\lambda + q \G_\lambda \in X \setminus H^1$ such that, up to subsequence,
\begin{itemize}
\item
$u_n \rightharpoonup u$ in $L^2$,
\item
$u_n \to u$ a.e. in $\real^3$,
\item
$\phi_{\lambda, n} \rightharpoonup \phi_\lambda$ in $H^1$ and
\item
$q_n \to q$ in $\complex$
\end{itemize}
as $n \to \infty$.
\end{itemize}
\end{lem}
\begin{proof}
As $\set{u_n}_{n \in \nat} \subset X_\mu$, there exists
$u \in L^2$ such that, up to subsequence, $u_n \rightharpoonup u$ in $L^2$ as $n \to \infty$. It follows from the previous lemma that there exist $C > 0$ and $\bar{n} > 0$ such that $\abs{q_n} > C$ for every
$n \geq \bar{n}$. In particular, we can associate each $n \in \nat$ to a $\phi_n \in H^1$ such that $u_n = \phi_n + q_n \G_{\nu_n}$, where
$\nu_n := \abs{q_n}^4 / \parens{64 \pi^2 \mu^2}$.

Notice that $\parens{\phi_n}_{n \in \nat}$ is bounded in $L^2$ because
\[
\norm{\phi_n}_{L^2}
\leq
\norm{u_n}_{L^2}
+
\abs{q_n} \norm{\G_{\nu_n}}_{L^2}
=
2 \sqrt{\mu}.
\]
It suffices to argue as in the proof of Lemma \ref{groundStates:lem:prop3.1} to deduce that $\parens{\phi_n}_{n \in \nat}$ and $\parens{q_n}_{n \in \nat}$ are respectively bounded in $D^{1, 2}$ and $\complex$.

Let us show that the lemma holds when
$
\lambda
\geq
1 + \parens{\sup_{n \in \nat} \abs{q_n}} / \parens{8 \pi \mu}^2
$.
Given $n \in \nat$, let
$\phi_{\lambda, n} = \phi_n + q_n \parens{\G_{\nu_n} - \G_\lambda}$,
so that $u_n = \phi_{\lambda, n} + q_n \G_\lambda$. Due to Lemma \ref{prelim:lem:Glambda}, we have
\[
\norm{\phi_{\lambda, n}}_{L^2}^2
\leq
2
\parens*{
	\norm{\phi_n}_{L^2}^2
	+
	\frac{1}{8 \pi}
	\parens*{
		\frac{1}{\sqrt{\lambda}}
		+
		\frac{1}{\sqrt{\nu_n}}
		-
		\frac{4}{\sqrt{\lambda} + \sqrt{\nu_n}}
	}
	\abs{q_n}^2
}
\]
and
\[
\norm{\phi_{\lambda, n}}_{D^{1, 2}}^2
\leq
2
\parens*{
	\norm{\phi_n}_{D^{1, 2}}^2
	+
	\frac{1}{8 \pi}
	\cdot
	\frac{
		3 \lambda \sqrt{\nu_n} - 3 \nu \sqrt{\lambda}
		+
		\nu_n \sqrt{\nu_n} - \lambda \sqrt{\lambda}
	}{\nu_n - \lambda}
	\abs{q_n}^2
},
\]
so $\parens{\phi_{\lambda, n}}_{n \in \nat}$ is bounded in $H^1$. We deduce that, up to subsequence,
$\phi_{\lambda, n} \rightharpoonup \phi_\lambda$ in $H^1$ as
$n \to \infty$. Furthermore, $u = \phi_\lambda + q \G_\lambda$. By the Kondrakov theorem, we have $\phi_{\lambda, n} \to \phi_n$ in $L^r_{loc}$ as $n \to \infty$ for every $2 \leq r < 6$, so $u_n \to u$ a.e. as $n \to \infty$.

A similar argument shows that the lemma also holds when
$
0
<
\lambda
<
1 + \parens{\sup_{n \in \nat} \abs{q_n}} / \parens{8 \pi \mu}^2
$.
\end{proof}

The last preliminary result is that $N$ satisfies the following BL-splitting property that follows from \cite[Lemma 2.4]{morozGroundstatesNonlinearChoquard2013}.

\begin{lem}
\label{groundStates:lem:BLSplittingForN}
Suppose that $1 \leq p < \infty$ and $\parens{u_n}_{n \in \nat}$ is a bounded sequence in $L^{\frac{6 p}{3 + \beta}}$ such that $u_n \to u$ a.e. as $n \to \infty$. We deduce that
\[N \parens{u_n} - N \parens{u_n - u} \to N \parens{u}\]
as $n \to \infty$.
\end{lem}

We finally proceed to the proof of the theorem.
\begin{proof}
[Proof of Theorem \ref{intro:thm:existence_of_ground_states}]
Suppose that
$\parens{u_n := \phi_{\lambda, n} + q_n \G_\lambda}_{n \in \nat}$ is a minimizing sequence for $E|_{X_\mu}$. In view of Lemma \ref{groundStates:lem:AdamiLem3.4}, we can suppose that $q_n \neq 0$ for every $n \in \nat$.

Take $u := \phi_\lambda + q \G_\lambda \in X \setminus H^1$ as furnished by Lemma \ref{groundStates:lem:AdamiLem3.5}; let
$m = \norm{u}_{L^2}^2$ and suppose that $u_n \rightharpoonup u$ in $L^2$, $u_n \to u$ a.e. in $\real^3$ as $n \to \infty$, which indeed holds up to subsequence. The norm $\norm{\cdot}_{L^2}$ is weakly lower semicontinuous, so $m \leq \mu$. Lemma \ref{groundStates:lem:AdamiLem3.4} implies $m > 0$, so
$0 < m \leq \mu$.

By contradiction, suppose that $0 < m < \mu$. As
$u_n \rightharpoonup u$ in $L^2$ as $n \to \infty$, we deduce that
$\norm{u_n - u}_{L^2}^2 \to \mu - m > 0$ as $n \to \infty$ and
$\mu / \norm{u_n - u}_{L^2}^2 > 1$ for sufficiently large $n$. Suppose from now on that $n$ is sufficiently large. On one hand,
\begin{align*}
\E \parens{\mu}
&\leq
E \parens*{
	\frac{\sqrt{\mu}}{\norm{u_n - u}_{L^2}} \parens{u_n - u}
};
\\
&\leq
\frac{1}{2} \frac{\mu}{\norm{u_n - u}_{L^2}^2}
Q \parens{u_n - u}
-
\frac{1}{2 p} \frac{\mu^p}{\norm{u_n - u}_{L^2}^{2 p}}
N \parens{u_n - u};
\\
&\leq
\frac{\mu}{\norm{u_n - u}_{L^2}^2}
E \parens{u_n - u}
\end{align*}
because $p > 1$. Therefore,
\[
\frac{\mu - m}{\mu} \E \parens{\mu}
\leq
\liminf_{n \to \infty} E \parens{u_n - u}.
\]
On the other hand, a similar argument shows that
\[
\frac{m}{\mu} \E \parens{\mu}
<
E \parens{u}.
\]
In view of these inequalities, we have
\[
\E \parens{\mu}
<
E \parens{u} + \liminf_{n \to \infty} E \parens{u_n - u}.
\]
It then follows from the BL Lemma and Lemma \ref{groundStates:lem:BLSplittingForN} that
$\E \parens{\mu} < \E \parens{\mu}$, which is a contradiction, so
$m = \mu$.

At this point, we already know that $u_n \to u$ in $L^2$,
$\phi_{\lambda, n} \rightharpoonup \phi_\lambda$ in $H^1$ and $q_n \to q$ in $\complex$ as $n \to \infty$. As such, it follows from Proposition \ref{groundStates:prop:generalizationOfYe} that
$N \parens{u_n - u} \to 0$ as $n \to \infty$. In view of this limit and Lemma \ref{groundStates:lem:BLSplittingForN}, we obtain
$E \parens{u} \leq \E \parens{\mu}$, hence the result.
\end{proof}

\sloppy
\printbibliography
\end{document}